\documentclass{amsart}

\usepackage{amssymb}

\usepackage{graphicx}

\newtheorem{theorem}{Theorem}[section]
\newtheorem{lemma}[theorem]{Lemma}

\theoremstyle{definition}
\newtheorem{corollary}[theorem]{Corollary}
\newtheorem{proposition}[theorem]{Proposition}
\newtheorem{assumption}[theorem]{Assumption}
\newtheorem{definition}[theorem]{Definition}
\newtheorem{example}[theorem]{Example}

\theoremstyle{remark}

\numberwithin{equation}{section}

\begin{document}

\title[On Maximal Ranges of Vector Measures for Subsets]{On Maximal Ranges of Vector Measures for Subsets and Purification
of Transition Probabilities}


\author{Peng Dai}
\address{Department of Applied Mathematics and Statistics, Stony Brook University, Stony Brook, NY 11794-3600}
\email{Peng.Dai@stonybrook.edu}

\author{Eugene A. Feinberg}
\address{Department of Applied Mathematics and Statistics, Stony Brook University, Stony Brook, NY 11794-3600}
\email{Eugene.Feinberg@sunysb.edu}
\thanks{This research was partially supported by NSF grants CMMI-0900206 and CMMI-0928490.}


\date{\today}


\commby{Richard C. Bradley}

\begin{abstract}
Consider a measurable space with an atomless finite vector
measure. This measure defines a mapping of the $\sigma$-field into
an Euclidean space. According to the Lyapunov convexity theorem,
the range of this mapping is a convex compactum. Similar ranges
are also defined for measurable subsets of the space. Two subsets
with the same vector measure may have different ranges. We
investigate the question whether, among all the subsets having the
same given vector measure, there always exists a set with the
maximal range of the vector measure. The answer to this question
is positive for two-dimensional vector measures and negative for
higher dimensions. We use the existence of maximal ranges to strengthen
the Dvoretzky-Wald-Wolfowitz purification theorem for the case of two measures.
\end{abstract}

\maketitle



%
%
%
%
%
%
%
%
%
%
%
%
%

\section{Introduction}\label{s1}

Let $\left({X},\mathcal{F}\right)$ be a measurable space
and ${\mu}=\left(\mu_{1},...,\mu_{m}\right)$, $m=1,2,\ldots,$
be a finite atomless vector measure on it.  We recall that a measure $\nu$ is
called atomless if for each ${Z}\in\mathcal{F},$ such that
$\nu\left({Z}\right)>0$, there exists
 ${Z}'\in\mathcal{F}$
such that  ${Z}'\subset{Z}$ and
$0<\nu\left({Z}'\right)<\nu\left({Z}\right)$. A vector measure
${\mu}=\left(\mu_{1},...,\mu_{m}\right)$, is called finite and
atomless if each measure $\mu_{i}$, $i=1\dots m$, is finite and
atomless. For each ${Y}\in \mathcal{F}$ consider the range
${R}_{{\mu}}\left({Y}\right)=\left\{
{\mu}\left({Z}\right):{Z}\in\mathcal{F},{Z}\subseteq{Y}\right\} $
of the vector measures of all its measurable subsets $Z$.
According to
 the Lyapunov convexity theorem~\cite{Liapounoff:1940}, the sets
 ${R}_{{\mu}}\left({Y}\right)$ are convex compactums in
 $\mathbb{R}^m.$  For a review on this theorem and its applications see \cite{Olech}.

In this paper we study whether for any $ p\in {R}_{{\mu}}(X)$, the
set of all ranges $\{{R}_{{\mu}}\left({Y}\right): \mu(Y)= p,
Y\in{\mathcal{F}}\}$ contains a maximal element. In other words,
is it always true that for any $ p\in {R}_{{\mu}}(X)$ there exists
a subset $Y^*\in \mathcal{F}$ such that $ \mu(Y^*)= p$ and
${R}_{{\mu}}\left({Y^*}\right)\supseteq
{R}_{{\mu}}\left({Y}\right)$ for any $Y\in \mathcal{F}$ with
$\mu(Y)= p$?  We  show that the answer is positive when $m=2$ and
negative when $m>2$. Furthermore, for $m=2$,  this maximal range
can be constructed by a simple geometric transformation of
${R}_{{\mu}}\left({X}\right).$

In addition to the maximal range, it is possible to consider a
minimal range. For $q\in R_\mu (X)$, the set $M^* \in \mathcal{F}$
with $\mu(M^*)=q$ has minimal range corresponding to $q$ if
$R_\mu(M) \supseteq R_\mu\left( M^* \right)$ for any $M\in
\mathcal{F}$ with $\mu(M)=q$. We show that a set has a maximal
range corresponding to $p$ if and only if its complement has a
minimal range corresponding to $\mu(X)-p$. Therefore, minimal
ranges also exist for dimension two and they may not exist for
higher dimensions.

Lyapunov's theorem is relevant to purification of  transition
probabilities discovered by Dvoretzky, Wald and Wolfowitz
\cite{Dvoretzky:1950, Dvoretzky:1951} for a finite image set. Let
$(A,{\mathcal{A}})$ be a measurable space and $\pi$ be a transition
probability from $X$ to $A$; that is, $\pi(B|x)$ is a measurable
function on $(X,{\mathcal{F}})$ for any $B\in \mathcal{A}$ and $\pi(\cdot|x)$
is a probability measure on $(A,{\mathcal{A}})$ for any $x\in X$.
According to Dvoretzky, Wald and Wolfowitz \cite{Dvoretzky:1950,
Dvoretzky:1951}, two transition probabilities $\pi_1$ and $\pi_2$
are called strongly equivalent if
\begin{equation}\label{eq:SETEQ}
\int_{{X}}\pi_{1}\left(B|x\right)\mu_i\left(dx\right)=
\int_{{X}}\pi_{2}\left(B|x\right)\mu_i\left(dx\right),
\qquad i=1,\dots,m, \quad B\in {\mathcal{A}}.\end{equation}

A transition probability $\pi$ is called pure if each measure
$\pi(\cdot|x)$ is concentrated at one point.  A pure transition
probability $\pi$ is defined by a measurable mapping $\varphi:X\to
A$ such that $\pi(B|x)=I\{\varphi(x)\in B\}$ for all $B\in
\mathcal{A}.$  According to the contemporary terminology, a transition
probability can be purified if for it there exists a strongly
equivalent pure transition probability.

For a finite set $A$, Dvoretzky, Wald and Wolfowitz
\cite{Dvoretzky:1950,Dvoretzky:1951} proved that any transition
probability can be purified (we recall that $\mu$ is atomless).
Edwards \cite[Theorem 4.5]{Edwards:1987} generalized this result
to the case of a countable set $A$. Khan and Rath \cite[Theorem
2]{Khan:2009} gave another proof of this generalization. Loeb and
Sun \cite[Example 2.7]{Loeb:2006} constructed an elegant example
when a transition probability cannot be purified for $m=2$,
$X=[0,1]$, and  $A=[-1,1]$.
However,  purification holds for a countable set of nonatomic,
finite, signed Loeb measures, when $A$ is a complete separable
metric space \cite[Corollary~2.6]{Loeb:2006}.

Note that for a countable (finite or infinite) set $A$, a transition probability $\pi$ can be purified if and only if
there exists a partition $\{Z^a\in {\mathcal{F}}:a\in A\}$ of $X$, such that
\begin{equation} \label{eq:atom}
\int_{{X}}\pi\left(a|x\right)\mu \left(dx\right) = \mu (Z^a),\qquad a\in{A}.
\end{equation}
where $\mu$ is a $m$-dimensional finite atomless vector measure.
Since $\int_{{X}}\pi\left(a|x\right)\mu \left(dx\right)$ are
vectors in $\mathbb{R}^m$, a natural question is: under what
conditions for an arbitrary set of vectors $\left\{p^a:a\in
A\right\}$ there exists a partition $\{Z^a\in {\mathcal{F}}:a\in
A\}$ of $X$  such that $p^a =\mu (Z^a)$ for each $a\in{A}$. We use
the theorem on maximal ranges proved in this paper to show that
for $m=2$, such partition exists if and only if (i)~$\sum_{a\in
A}p^a=\mu(X)$, and (ii)~$\sum_{a\in B}p^a \in R_\mu (X)$ for any
finite subset $B$ of $A$. For $m=2$, the Dvoretzky-Wald-Wolfowitz
theorem for a countable set $A$ \cite{Edwards:1987,Khan:2009}
follows from the sufficient part of this statement.


%
%


We formulate the main results in the following section, prove the
existence of maximal and minimal subsets for $m=2$ in
Section~\ref{s3}, provide counterexamples when $m>2$ in
Section~\ref{s4}, and describe geometric constructions of maximal
ranges in Section~\ref{s5}.  Section~\ref{s6} is devoted to the
proof of the theorem on the existence of a partition.

\section{Main results}\label{s2}

\begin{definition}\label{def:sets}

Given a measurable subset ${Y}$ of the measurable space $\left({X},\mathcal{F}\right)$
with a vector measure ${\mu}$ and a vector ${p}\in{R}_{{\mu}}\left({Y}\right)$, we define

(a) the set of all subsets of ${Y}$ with vector measure
${p}$,
\[
\mathcal{S}^{{p}}_\mu\left({Y}\right)=\left\{
{Z}\in\mathcal{F}_Y:
{\mu}\left({Z}\right)={p}\right\} ,\] where ${\mathcal{
F}}_Y=\{Z\subseteq Y: Z\in{\mathcal{F}}\};$

(b) the union of all the ranges of all subsets of ${Y}$
with the vector measure ${p}$,
\[
{R}^{{p}}_{\mu} \left({Y}\right)=\bigcup \limits
_{{Z}\in\mathcal{S}^{{p}}_\mu\left({Y}\right)}{R}_\mu\left({Z}\right);\]

(c) the intersection of the $R_\mu(Y)$ with its shift
by a vector $- ({ \mu(Y) -p})$,
\[
Q^{{p}}_{\mu}\left({Y}\right)=\left({R}_{{\mu}}\left({Y}\right)-\left\{
{\mu}\left({Y}\right)-{p}\right\}
\right)\cap{R}_{{\mu}}\left({Y}\right),
\]
where
${S}_1-{S}_2=\{ {q}-{r}:{q}\in {S}_1,{r}\in{S}_2\} $ for ${S}_1,{S}_2\in \mathbb{R}^m$. In particular,
${R}_{{\mu}}\left({Y}\right)-\{{r}\}$ is a
parallel shift of ${R}_{{\mu}}\left({Y}\right)$ by
$-{r}$.
\end{definition}

\begin{definition}\label{def:maxset}
For a measurable subset ${Y}\in\mathcal{F}$, the set
${Z}^{*}\in\mathcal{S}^{{p}}_\mu\left({Y}\right)$,
such that
${R}_\mu\left({Z}^{*}\right)={R}^{{p}}_\mu\left({Y}\right)$,
is called the maximal subset of $Y$ with the measure $p$. The set $M^*\in\mathcal{S}^{{q}}_\mu\left({Y}\right)$, such that
${R}_\mu\left({M}^{*}\right) \subseteq {R}_\mu \left({M}\right)$ for any $M\in\mathcal{S}^{{q}}_\mu\left({Y}\right)$, is called the minimal subset of $Y$ with the measure $q$.
\end{definition}

Our first result is the following theorem.

\begin{theorem}\label{thm1}
For a two-dimensional finite atomless vector measure
${\mu}=\left(\mu_{1},\mu_{2}\right)$ and for a vector
${p}\in{R}_{{\mu}}\left({X}\right)$, there exists a maximal set
${Z}^{*}\in\mathcal S_\mu^{{p}}\left({X}\right)$ and, in addition,
${R}^{{p}}_\mu\left({X}\right)=
Q^{{p}}_{\mu}\left({X}\right).$
\end{theorem}

Theorem \ref{thm1} immediately implies that the set
${R}^{{p}}_\mu\left({X}\right)$, which is the union
of the ranges of ${\mu}$ on ${Z}$, for all
${Z}\in\mathcal{S}^{{p}}_\mu\left({X}\right)$, is a
convex compactum. Furthermore, if ${R}_{{\mu}}\left({X}\right)$ and $ p$ are
given, the set ${R}^{ p}_\mu(X)$ is defined by two simple
geometric operations, a shift and an intersection,  since
$Q^{{p}}_{\mu}\left({X}\right)$ is defined by
these operations.

The following theorem links the notions maximal and minimal subsets.
\begin{theorem}\label{thm1.5}
The set $Z^*$ is the maximal subset of $X$ with the measure $p$, if and only if $M^*=X \setminus Z^*$ is the minimal subset of $X$ with the measure $\mu(X) - p$.
\end{theorem}

We will use Theorem \ref{thm1} to prove the following theorem
that, as shown in Section~\ref{s6}, strengthens the
Dvoretzky-Wald-Wolfowitz purification theorem
\cite{Edwards:1987,Khan:2009} for the case $m=2.$

\begin{theorem}\label{thm2}
Consider a measurable space $\left({X},\mathcal{F}\right)$ with a
two-dimensional finite atomless vector measure ${\mu}$, a
countable set $A$, and a set of two-dimensional vectors
$\left\{p^a:a\in A\right\}$. A partition $\{Z^a\in
{\mathcal{F}}:a\in A\}$ of $X$, with $p^a =\mu (Z^a)$ for all
$a\in{A}$, exists if and only if (i) $\sum_{a \in A}  {p^a} =
\mu(X)$ and (ii) $\sum_{a \in B} {p^a} \in R_\mu (X)$ for any
finite subset $B \subset A$.
\end{theorem}

\section{Maximal and minimal subsets}\label{s3}

In this section, we prove Theorems \ref{thm1} and \ref{thm1.5}. Recall that for a set ${S}\subseteq \mathbb{R}^m$, its reflection across a point ${c}\in \mathbb{R}^m$ is ${\textrm{Ref}}({S},{c})=\{2{c}\}- {S}$.
If ${S}=\{{s}\}$ is a singleton, we shall write ${\textrm{Ref}}({s},{c})$ instead of ${\textrm{Ref}}(\{{s}\},{c})$.
A set ${S} \subseteq \mathbb{R}^{m}$ is called centrally symmetric if $ {\textrm{Ref}}\left({S},{c}\right) = {S} $ for some point ${c}\in \mathbb{R}^m$ called the center of ${S}$.
Any bounded centrally symmetric set has only one center.

In this section we let ${Y}\in\mathcal{F}$ be any measurable subset of $X$. Lemmas \ref{lem:censym}-\ref{lem:censymsub} hold for any finite atomless vector measure ${\mu}=(\mu_1,\ldots,\mu_m)$ on $(X,{\mathcal F})$, where $m=1,2,\ldots\ $.

\begin{lemma}\label{lem:censym}
The set ${{R}_\mu}({Y})$ is centrally symmetric with the center $\frac{1}{2}{\mu}\left({Y}\right)$.
\end{lemma}
\begin{proof}
The proof is straightforward, and this fact was observed by Lyapunov \cite[p. 476]{Liapounoff:1940}.
\end{proof}

\begin{lemma}\label{lem:translation}
The equality ${R}_{{\mu}}\left({Y}\right)-\left\{
{\mu}\left({Y}\right)-{p}\right\}
={\textrm{Ref}}\left({R}_{{\mu}}\left({Y}\right),\frac{1}{2}{p}\right)
$ holds for any ${p}\in{R}_{{\mu}}\left({Y}\right)$.
\end{lemma}
\begin{proof}
By Lemma~\ref{lem:censym}, ${R}_{{\mu}}\left({Y}\right)={\textrm{Ref}}\left({R}_{{\mu}}\left({Y}\right)
,\frac{1}{2}{\mu}\left({Y}\right)\right)=\{{\mu}(Y)\}-{R}_{{\mu}}\left({Y}\right)$.
Therefore, ${R}_{{\mu}}\left({Y}\right)-\left\{
{\mu}\left({Y}\right)-{p}\right\}=(\{ {\mu}\left({Y}\right)\}-{R}_{{\mu}}\left({Y}\right))- \left\{
{\mu}\left({Y}\right)-{p}\right\}=\{{{p}}\}-{R}_{{\mu}}\left({Y}\right)={\textrm{Ref}}\left({R}_{{\mu}}\left({Y}\right),\frac{1}{2}{p}\right).$
\end{proof}

\begin{lemma}\label{lem:censymsub}
Each of the sets ${R}_\mu^{{p}}\left({Y}\right)$ and $Q_\mu^{{p}}\left({Y}\right)$ is centrally symmetric with the center $\frac{1}{2}{p}$.
\end{lemma}
\begin{proof}
According to Lemma \ref{lem:censym}, each set $ Z \in
\mathcal{S}^{{p}}_\mu(Y)$ is centrally symmetric with the center
$\frac{1}{2} {p}$. Therefore,
${R}_\mu^{{p}}\left({Y}\right)$, which is  the union of all the
sets in $Z \in \mathcal{S}^{{p}}_\mu(Y)$, is also centrally
symmetric with the center $\frac12 {p}$.

In addition,
\begin{eqnarray*}
{\textrm{Ref}}\left(Q^{{p}}_{\mu}\left({Y}\right),\frac{1}{2}{p}\right)
&=&{\textrm{Ref}}\left(\left({R}_{{\mu}}\left({Y}\right)-\left\{
{\mu}\left({Y}\right)-{p}\right\}
\right)\cap{R}_{{\mu}}\left({Y}\right),\frac{1}{2}{p}\right)\\
&=&{\textrm{Ref}}\left({\textrm{Ref}}\left({R}_{{\mu}}\left({Y}\right),\frac{1}{2}{p}\right)\cap{R}_{{\mu}}\left({Y}\right),\frac{1}{2}{p}\right)\\
&=&{R}_{{\mu}}\left({Y}\right)\cap{\textrm{Ref}}\left({R}_{{\mu}}\left({Y}\right),\frac{1}{2}{p}\right) \\
&=&{R}_{{\mu}}\left({Y}\right)\cap\left({R}_{{\mu}}\left({Y}\right)-\left\{{\mu}\left({Y}\right)-{p}\right\}\right)
=Q^{{p}}_{\mu}\left({Y}\right),
\end{eqnarray*}
where the first and last
equalities follow from the definition of $Q^{{p}}_{\mu},$ the
second and second to the last equalities follow from
Lemma~\ref{lem:translation}. The third equality holds because a
reflection of intersections equals the intersection of reflections
and, in addition, a reflection of a reflection across the same
point is the original set.
\end{proof}

Here we present the major ideas of the proof of
Theorem~\ref{thm1}.  First, as shown later, after
Theorem~\ref{thm1} is proven for equivalent measures $\mu_1$ and
$\mu_2$, this condition can be removed.  So, we make the following
assumption in Lemmas~\ref{lem:l_q},
\ref{lem:L_q}-\ref{lem:weakthm}.
\begin{assumption}\label{ass:abscon}
The measures $\mu_{1}$ and $\mu_{2}$ are finite, atomless, and
equivalent.
\end{assumption}


%
%
%
%
Under Assumption \ref{ass:abscon}, let
$f\left(x\right)=\frac{d\mu_{2}}{d\mu_{1}}(x)$ be a Radon-Nikodym
derivative of $\mu_2$ with respect to $\mu_1$. Since $f$ is
defined $\mu_1$-{\it a.e.}, we fix any its version. We shall
frequently use notations similar to
\[
\left\{ f(x)<l \right\} =\left\{ x\in{X} : f(x)<l\right\}.
\]

Second, under Assumption \ref{ass:abscon}, for any
$a\in\left[0,\mu_1\left(X\right)\right]$, we
denote 
\begin{equation}\label{eq:lq}
l_{a}=\min\left\{ l\ge0:\mu_{1}\left(\left\{ f\left(x\right)\le
l\right\} \right)\ge a\right\}.
\end{equation}
Observe that the minimum in
(\ref{eq:lq}) exists.  Indeed, let
\[
l_{a}=\inf\left\{l\ge0:\mu_{1}\left(\left\{ f\left(x\right)\le l\right\} \right)\ge a\right\}.
\]
We need to show that
$\mu_{1}\left(\left\{ f\left(x\right)\le l_a\right\} \right)\ge a.$
If $l_a=\infty$ then $\mu_1 (\{f(x)\le \infty\}) = \mu_1(X) \ge a$. Let $l_a<\infty$.
Consider a sequence $l^k\searrow l_a$, $k=1,2,\dots\ $.  Then
$\cap_{k=1}^\infty \{f(x)\le l^k\}=\{f(x)\le l_a\}$ and $\{f(x)\le
l^k\}\supseteq \{f(x)\le l^{k+1}\}$.  Therefore $\mu_1(\{f(x)\le
l_a\})=\lim_{k\to\infty}\mu_1(\{f(x)\le l^k\})\ge a.$

Third, it is possible to construct the maximal set $Z^*=X\setminus M^*$, where $M^*$ can be defined explicitly. Let ${X}^{l}=\left\{f\left(x\right)=l\right\}$. If
$\mu_1(X^l)=0$ for all $l\in [0,\infty)$, then the definition of $M^*$ is easier and we explain it first.  In this case,
there exists $a^*\in [0,\mu_1(X)]$ such that
$\mu_2\left(M^*\right)=\mu_2(X)-p_2$, and $M^*$ can be defined as
\begin{equation}\label{eq:M_star}
M^*=\{l_{a^*} \le y<l_{a^*+(\mu_1(X)-p_1)}\}.
\end{equation}
In the general situation, the number $a^*$ can be chosen to satisfy
\[
\nonumber \mu_2(\{l_{a^*}<y<l_{a^*+(\mu_1(X)-p_1)}\})\le\mu_2(X)-p_2\le\mu_2(\{l_{a^*}\le y\le l_{a^*+(\mu_1(X)-p_1)}\}),
\]
and
\begin{equation} \label{eq:gM_star}
M^*=\{l_{a^*}<y<l_{a^*+(\mu_1(X)-p_1)}\}\cup Z^1\cup Z^2,
\end{equation}
for $Z^i$, $i=1,2$, being some measurable subsets of $X^{l^i}$,
where $l^1=l_{a^*}$ and $l^2=l_{a^*+(\mu_1(X)-p_1)}$.  In
particular, if $\mu_1\left(X^{l^1}\right)=0$, let $Z^1 = X^{l^1}$, and if $\mu_1\left(X^{l^2}\right)=0$, let $Z^2 = \emptyset$. If $\mu_1\left(X^{l^1}\right)=\mu_1\left(X^{l^2}\right)=0$, then (\ref{eq:gM_star}) reduces to (\ref{eq:M_star}). It is easy to show that the number of $l$ such that $\mu_1\left(X^l\right)=0$ is countable, but we do not use this fact.

The proof of Theorem~\ref{thm1} is based on several lemmas.

\begin{lemma}\label{lem:l_q}
Under Assumption \ref{ass:abscon}, the numbers $l_a$, $a\in\left[0,\mu_1(X) \right]$ have the following properties: (a) $\mu_{1}\left(\left\{ f\left(x\right)<l_{a}\right\} \right)\le a\le\mu_{1}\left(\left\{ f\left(x\right)\le l_{a}\right\} \right)$; (b) $l_a \le l_{a'}$ if $a \le a'$.
\end{lemma}

\begin{proof}
For (a), by definition, $a\le\mu_{1}\left(\left\{ f\left(x\right)\le l_{a}\right\} \right)$. To prove that $\mu_{1}\left(\left\{ f\left(x\right)<l_{a}\right\} \right)\le a$, assume that $\mu_{1}\left(\left\{ f\left(x\right)<l_{a}\right\} \right) > a $. If $l_a=0$, then $\mu_{1}\left(\left\{ f\left(x\right)<l_{a}\right\} \right) = 0 > a $ which contradicts the assumption that $a \ge 0$. If $l_a>0$, let $\epsilon_k \searrow 0$, $k=1,2,\dots\ $, be a sequence of positive numbers. Then, for $k=1,2,\dots$,
\[
\mu_1\left( \left\{f(x)<l_a \right\}\right) = \mu_1\left( \left\{f(x) \le l_a-\epsilon_k \right\}\right) + \mu_1\left( \left\{l_a-\epsilon_k < f(x) < l_a \right\} \right) > a.
\]
Let $D_k= \left\{ l_a-\epsilon_k < f(x) < l_a \right\}$. We observe that $D_{k+1} \subseteq D_k$ and $\cap_{u=1}^\infty D_k =\emptyset$. Therefore, $\lim_{k\rightarrow \infty} \mu_1\left(D_k\right) = 0$. Thus,  $\mu_1\left( f(x) \le l_a-\epsilon \right)>a$ for some $\epsilon>0$ and this contradicts (\ref{eq:lq}). These contradictions imply the lemma.

For (b), assume $l_a > l_{a'}$, then $\mu_{1}\left(\left\{ f\left(x\right)\le
l_{a'}\right\} \right)\ge a' \ge a$, and this contradicts (\ref{eq:lq}).
\end{proof}

Note that for each $l\in[0,\infty)$, there exists a subfamily
\[
\left\{W_b\left({X^l}\right) \in \mathcal{F}_{{X^l}}: b\in\left[0,\mu_1 \left({X^l}\right)\right]\right\}
\]
such that ${W}_{b}\left({X^l}\right)\subset
{W}_{b'}\left({X^l}\right)\subseteq{X^l}$ whenever $b <
b'\le\mu_{1}\left({X^l}\right)$ and
$\mu_{1}\left({W}_{b}\left({X^l}\right)\right)=b$ for each
$b\in\left[0,\mu_1 \left({X^l}\right)\right]$. This fact follows
from Ross \cite[Theorem 2(LT3)]{Ross:2005}. We set
$W_0\left(X^l\right)=\emptyset$. From now on we fix a family of
${W}_{b}\left({X^l}\right)$ for each $l \in [0,\infty)$.

\begin{definition} \label{def:L_q}
Under Assumption \ref{ass:abscon}, for each $a$, define the following set
\begin{equation}\label{eq:L_a}
{L}_{a}=
\left\{ f\left(x\right)<l_{a}\right\} \cup{W}_{c}\left(X^{l_a}\right),
\end{equation}
where $c={a-\mu_{1}\left(\left\{ f\left(x\right)<l_{a}\right\} \right)}$.
\end{definition}
Note that property (a) in Lemma \ref{lem:l_q} guarantees that $c \in\left[0,\mu_1 \left({X^l}\right)\right]$.
%
\begin{lemma} \label{lem:L_q}
Under Assumption \ref{ass:abscon}, the sets $L_a\in\mathcal{F}$, $a\in\left[0,\mu_1(X) \right]$, have the following properties:
(a) $\mu_{1}\left({L}_{a}\right)=a$; (b) $\left\{ f\left(x\right)<l_{a}\right\} \subseteq{L}_{a}\subseteq\left\{ f\left(x\right)\le l_{a}\right\} $;
 (c) ${L}_{a}\subset{L}_{a'}\subseteq{X}$ if $a < a'\le\mu_{1}\left({X}\right)$.
\end{lemma}  

\begin{proof}

For (a),
$\mu_{1}\left({L}_{a}\right)  =  \mu_{1}\left(\left\{ f\left(x\right)<l_{a}\right\} \right)+\mu_{1}\left(W_{c}\left(X^{l_{a}}\right)\right)
  = \mu_{1}\left(\left\{ f\left(x\right)<l_{a}\right\} \right)+a-\mu_{1}\left(\left\{ f\left(x\right)<l_{a}\right\} \right)=a.$
Property (b) follows from ${W}_{c}\left(X^{l_a}\right) \subseteq X^{l_a}$ and (\ref{eq:L_a}).
For (c), if $l_a=l_{a'}$ then $c<c'$ where $c'={a'-\mu_{1}\left(\left\{ f\left(x\right)<l_{a}\right\} \right)}$, and thus
\[
L_a = \left\{ f\left(x\right)<l_{a}\right\} \cup{W}_{c}\left(X^{l_a}\right) \subset  \left\{ f\left(x\right)<l_{a}\right\} \cup{W}_{c'}\left(X^{l_a}\right) = L_{a'}.
\]
If $l_a < l_{a'}$ then $L_a\subseteq \left\{ f\left(x\right)\le l_{a}\right\} \subset \left\{ f\left(x\right)<l_{a'}\right\} \subseteq L_{a'}$.
\end{proof}

 Let ${M}_{a,d}={L}_{a+d}\setminus{L}_{a}$. For each
$d\in\left[0,\mu_{1}\left({X}\right)\right]$ and each
$a\in\left[0,\mu_{1}\left({X}\right)-d\right]$, denote
$g_{d}\left(a\right)=\mu_{2}\left({M}_{a,d}\right)=\int_{M_{a,d}}
f(x) \mu_1 (dx)$. The function $g_d(a)$ is non-decreasing and
continuous in $a\in\left[0,\mu_{1}\left({X}\right)-d\right]$ for
each $d \in \left[0,\mu_1(X)\right]$. However, we will not use the
fact that it is non-decreasing. So we only prove the continuity in
the following lemma.
\begin{lemma}
\label{lem:cont} Under Assumption \ref{ass:abscon},
$g_{d}\left(a\right)$ is continuous in
$a\in\left[0,\mu_{1}\left({X}\right)-d\right]$ for each $d \in
\left[0,\mu_1(X)\right]$.
\end{lemma}
\begin{proof}
We show that $\mu_2(L_{a+d})$  is continuous in $a \in
\left([0,\mu_1(X)-d\right]$ for any $d \in
\left[0,\mu_1(X)\right]$. Since $g_d(a)=L_{a+d}-L_a$, this
implies the lemma. Consider a sequence $\{a^k: k=1,2,\ldots\},$
where $a_k\in [0,\mu_{1}({X})-d]$. Let $a^k\nearrow a.$  Then
$L_{a^k+d}\subset L_{a^{k+1}+d}\subset B\subseteq L_{a+d},$ where
$B=\cup_{k=1}^\infty L_{a^k+d}.$ Therefore,
$\mu_i(L_{a^k+d})\nearrow\mu_i(B)$ and
$\mu_i(L_{a+d})=\mu_i(B)+\mu_i(L_{a+d}\setminus B),$ $i=1,2.$
Since $\mu_1(L_{a^k+d})=a^k+d\nearrow a+d=\mu_1(L_{a+d}),$ we have
$\mu_1(B)=a+d$ and $\mu_1(L_{a+d}\setminus B)=0.$  Since $\mu_1$
and $\mu_2$ are equivalent measures, $\mu_2(L_{a+d}\setminus B)=0$
and $\mu_2(L_{a^k+d})\nearrow \mu_2(L_{a+d}).$

Now let $a^k\searrow a.$  Then $L_{a^k+d}\supset
L_{a^{k+1}+d}\supset D\supseteq L_{a+d},$ where
$D=\cap_{k=1}^\infty L_{a^k+d},$ and
$\mu_i(L_{a^k+d})\searrow\mu_i(D)$,
$\mu_i(L_{a+d})=\mu_i(D)-\mu_i(D\setminus L_{a+d})$ for $i=1,2.$
Similar to the previous case, $\mu_1(L_{a^k+d})=a^k+d\searrow
a+d=\mu_1(L_{a+d}),$ so $\mu_1(D)=a+d,$ $\mu_2(D\setminus
L_{a+d})= \mu_1(D\setminus L_{a+d})=0$, and $\mu_2(D)=\mu_2(L_{a+d}).$
Thus, $\mu_2(L_{a^k+d})\searrow \mu_2(L_{a+d}).$
\end{proof}

 Observe that a point ${q}\in \mathbb{R}^2$  is  on the upper (lower) boundary of ${R}_{{\mu}}\left({X}\right)$,
  if and only if ${q}\in {R}_{{\mu}}\left({X}\right)$ and $q'_{2}\le q_{2}$
   ($q'_{2}\ge q_{2}$) for any ${q}'\in{R}_{{\mu}}\left({X}\right)$ with $q'_{1}=q_{1}$. 

\begin{lemma}\label{lem:onboundary}\label{lem:infamily}
Under Assumption \ref{ass:abscon}, a point ${q}\in \mathbb{R}^2$
is on the lower boundary of ${R}_{{\mu}}\left({X}\right)$ if and
only if $0\le q_1\le\mu_1(X)$ and
$q_2={\mu}_2\left({L}_{q_{1}}\right)$, and it is on the upper
boundary of ${R}_{{\mu}}\left({X}\right)$ if and only if $0\le
q_1\le\mu_1(X)$ and
$q_2={\mu}_2\left({X}\setminus{L}_{\mu_{1}\left({X}\right)-q_{1}}\right)$.
\end{lemma}
\begin{proof}
For the lower boundary, let $q_2=\mu_2(L_{q_1}).$  Since
$q_1=\mu_1(L_{q_1}),$ we have $q=\mu(L_q)\in {R}_\mu(X).$ For any
set ${Z}\in\mathcal{F}$ with $\mu_{1}\left({Z}\right)=q_{1}$,
define disjoint sets ${Z}_{1}={Z}\setminus{L}_{q_{1}}$,
${Z}_{2}={L}_{q_{1}}\setminus{Z}$, and $M={Z} \cap {L}_{q_{1}}$.
 Then $Z=Z_1\cup M$,
 ${L}_{q_{1}}=Z_2\cup M$, and $\mu_1(Z_1)=\mu_1(Z_2)$, since $\mu_1(Z)=q_1=\mu_1(L_{q_1})$.
 Furthermore, $Z_1\subseteq
 \left\{f(x) \ge l_{q_{1}} \right\}$ and
 $Z_2\subseteq \left\{ f(x) \le  l_{q_{1}} \right\}$. Therefore,
\begin{eqnarray*}
\mu_{2}\left({Z}_{1}\right) & = & \int_{{Z}_{1}}f\left(x\right)\mu_{1}\left(dx\right)
 \ge  l_{q_{1}}\int_{{Z}_{1}}\mu_{1}\left(dx\right) \\
 &=&l_{q_{1}}\int_{{Z}_{2}}\mu_{1}\left(dx\right)
 \ge  \int_{{Z}_{2}}f\left(x\right)\mu_{1}\left(dx\right)=\mu_{2}\left({Z}_{2}\right).
 \end{eqnarray*}
So $\mu_{2}\left({Z}\right)= \mu_{2}\left({Z_1}\right) + \mu_{2}\left({M}\right) \ge \mu_{2} \left(Z_2\right)  + \mu_{2} (M)  = \mu_{2} \left({L}_{q_{1}}\right)$,
and thus ${q}$ is on the lower boundary of ${R}_{{\mu}}\left({X}\right)$.

If ${q}$ is on the lower boundary of
${R}_{{\mu}}\left({X}\right)$, then
$q_{2}\le\mu_{2}\left({L}_{q_{1}}\right)$. Since ${q}\in
{R}_{{\mu}}\left({X}\right)$, there exists ${Z}\in\mathcal{F}$
with ${\mu}\left({Z}\right)={q}$. But, as proved above,
$\mu_{2}\left({Z}\right)\ge\mu_{2}\left({L}_{q_{1}}\right)$  for
any ${Z}\in\mathcal{F}$ with $\mu_{1}\left({Z}\right)=q_{1}$. Thus
$q_{2}\ge\mu_{2}\left({L}_{q_{1}}\right)$. Therefore,
$q_{2}=\mu_{2}\left({L}_{q_{1}}\right)$.

The statement on the upper boundary follows from the symmetry of
${R}_{{\mu}}\left({X}\right)$.
\end{proof}

\begin{lemma}\label{lem:Z_star}

Under Assumption \ref{ass:abscon}, given ${u}=\left(u_{1},u_{2}\right)\in{R}_{{\mu}}\left({X}\right)$,
there exists $a^{*}\in\left[0,\mu_{1}\left({X}\right)-u_{1}\right]$
such that ${\mu}\left({M}_{a^{*},u_{1}}\right)={u}$.

\end{lemma}

\begin{proof}
Since $\mu_1(L_0)=0$ and $\mu_1$ and $\mu_2$ are equivalent,
$\mu_2(L_0)=0$.  Therefore, $ g_{u_{1}}(0)=\mu_2\left(L_{u_1}
\setminus L_0 \right) = \mu_2\left( L_{u_1} \right)-\mu_2\left(
L_{0} \right) = \mu_2\left( L_{u_1} \right). $ Similarly,
$\mu_2\left(X \setminus L_{\mu_1(X)}\right) = 0$, because
$\mu_1\left(X \setminus L_{\mu_1(X)}\right) =\mu_1\left(X\right) -
\mu_1\left(L_{\mu_1(X)}\right) = 0$.  Thus
\begin{eqnarray*}
g_{u_1}\left(\mu_1(X)-u_1\right) &=& \mu_2\left(L_{\mu_1(X)}
\setminus L_{\mu_1(X)-u_1} \right)
= \mu_2\left(L_{\mu_1(X)}\right) - \mu_2 \left(L_{\mu_1(X)-u_1} \right) \\
 &=& \mu_2\left(X\right) - \mu_2\left(X \setminus L_{\mu_1(X)}\right) - \mu_2 \left(L_{\mu_1(X)-u_1} \right) \\
 &=& \mu_2\left(X \setminus L_{\mu_1(X)-u_1} \right).
\end{eqnarray*}
According to Lemma \ref{lem:onboundary}, the
 point $\left(u_1,g_{u_{1}}\left(0\right)\right)$ is on the lower boundary of the range
 ${R}_{{\mu}}\left({X}\right)$ and the point $\left(u_1,g_{u_{1}}\left(\mu_{1}\left({X}\right)-u_{1}\right)\right)$
 is on the upper boundary of the range ${R}_{{\mu}}\left({X}\right)$. So
 $u_{2}\in\left[g_{u_{1}}\left(0\right),g_{u_{1}}\left(\mu_{1}\left({X}\right)-u_{1}\right)\right]$.
Since $g_{u_{1}}\left(a\right)$ is continuous in
$a\in\left[0,\mu_{1}\left({X}\right)-u_{1}\right]$, there exists
$a^{*}$, such that $g_{u_{1}}\left(a^{*}\right)=u_{2}$. That is,
$\mu_{2}\left({M}_{a^{*},u_{1}}\right)=u_{2}$. By definition,
$\mu_{1}\left({M}_{a^{*},u_{1}}\right)=u_{1}$. Therefore,
${\mu}\left({M}_{a^{*},u_{1}}\right)={u}$.
\end{proof}

Note that Lemmas \ref{lem:l_q}, and \ref{lem:L_q}-\ref{lem:Z_star} hold if one
replaces everywhere the set $X$ with any measurable subset $Z \in
\mathcal{F}$. In particular, expressions such as $\{f(x)<l\}$
should be replaced with $\{x\in Z: f(x)<l\}$. We define explicitly
\begin{equation} \label{eq:l_a(Y)}
l_{a}(Z)=\min\left\{ l\ge0:\mu_{1}\left(\left\{ x \in Z: f\left(x\right)\le
l\right\} \right)\ge a\right\}.
\end{equation}
Let $Z^l = \{ x \in Z : f(x)=l \}$. As follows from Ross
\cite[Theorem 2(LT3)]{Ross:2005}, for each $l\in[0,\infty)$, there
exists a family
\[
\left\{W_b\left({Z^l}\right) \in \mathcal{F}_{{Z^l}}: b\in\left[0,\mu_1 \left({Z^l}\right)\right]\right\}
\]
such that ${W}_{b}\left({Z^l}\right)\subset
{W}_{b'}\left({Z^l}\right)\subseteq{Z^l}$ whenever $b <
b'\le\mu_{1}\left({Z^l}\right)$ and
$\mu_{1}\left({W}_{b}\left({Z^l}\right)\right)=b$ for each
$b\in\left[0,\mu_1 \left({Z^l}\right)\right]$. Again, we fix a
family of ${W}_{b}\left({Z^l}\right)$ for each $l \in [0,\infty)$
and each $Z$, and define
\[
{L}_{a}(Z)= \left\{ x \in Z : f\left(x\right)<l_{a}\right\}
\cup{W}_{c}\left(Z^{l_a}\right),
\]
where $c=a-\mu_1 (\{x\in Z : f(x) <a\})$.
Note that $l_a(X)=l_a$ and $L_a(X)=L_a$, for each
$a\in[0,\mu_1(X)]$.  In the following two
lemmas and their proofs, for a given $u\in R_\mu(X)$,  we consider a point
$a^*\in\left[0,\mu_1(X)-u_1\right]$ with $\mu\left(
{M}_{a^{*},u_{1}} \right) = u$ and the set $Z=X\setminus
M_{a^*,u_1}$.  The existence of $a^*$ is stated in Lemma
\ref{lem:Z_star}. Later it will become clear that that $Z$ is the
maximal subset with the vector measure $p=\mu(X)-u$ and
$M_{a^*,u_1}$ is the the minimal subset with the vector measure
$p=u$.

\begin{lemma}\label{lem:connection}
Let Assumption \ref{ass:abscon} hold. For a given
$u=\left(u_1,u_2\right)\in R_\mu(X)$, consider
$a^*\in\left[0,\mu_1(X)-u_1\right]$ with $\mu\left(
{M}_{a^{*},u_{1}} \right) = u$.
  Then
\begin{equation}\label{eq:newL}
\mu_2\left(L_{a}(Z)\right)=
\begin{cases}
\mu_2\left({L}_{a}\right), & \textrm{if } a\in\left[0,a^{*}\right];\\
\mu_2\left({L}_{a+u_{1}}\setminus{M}_{a^{*},u_{1}}\right), & \textrm{if }
a\in\left(a^{*},{\mu_1}\left({X}\right)-u_{1}\right].
\end{cases}
\end{equation}
\end{lemma}
\begin{proof}
First, consider the case $a\in\left[0,a^{*}\right]$. We have $Z=X
\setminus M_{a^*,u_1} \supseteq L_{a^*} \supseteq L_{a} =
\left\{f(x)<l_a\right\} \cup W_c\left(X^{l_{a}}\right)$, where $c
= a - \mu_1 \left( \left\{ f(x) < l_a \right\} \right)$. In
addition, $\left\{f(x)<l_a\right\} \cup W_c\left(X^{l_{a}}\right)
\subseteq \left\{f(x) \le l_a\right\} $. Therefore
\begin{eqnarray*}
&&\mu_1 \left( \left\{x \in Z: f(x) \le l_a \right\} \right)
=\mu_1(Z\cap\{f(x)\le l_a\})\\ &\ge& \mu_1 \left(Z\cap( \left\{
f(x)<l_a\right\} \cup W_c(X^{l_a})) \right) = \mu_1 \left(
\left\{f(x)<l_a\right\} \cup W_c(X^{l_a}) \right) = a.
\end{eqnarray*}
Thus, (\ref{eq:l_a(Y)}) implies that $l_a(Z) \le l_a$. On the
other hand, take an arbitrary $l<l_a$. Since $Z \subseteq X$,
\begin{eqnarray*}
\mu_1 \left( \left\{x \in Z: f(x) \le l \right\} \right)\le \mu_1
\left( \left\{ f(x) \le l \right\} \right)<a.
\end{eqnarray*}
Therefore, $l_a(Z)>l$ for all $l<l_a.$  Thus, $l_a(Z)\ge l_a.$  We
conclude that $l_a(Z)=l_a.$

Denote $A=\{ f(x)<l_a\}.$ Since $Z\supseteq L_a\supseteq A$ and $l_a(Z)=l_a$, then
$\{x\in Z:\, f(x)<l_a(Z)\}= A.$ By definition, each of the sets $L_a$ and $L_a(Z)$ is the union of two
disjoint subsets: $L_a= A\cup W_c(X^{l_a})$ and $L_a(Z)=A\cup
W_b(Z^{l_a})$ with $c = a - \mu_1(A) = b$. Thus, since
$X^{l_a}\supseteq Z^{l_a}$ and $f(x)=l_a$ when $x\in X^{l_a}$, we
have $\mu_2(W_c(X^{l_a}))= \mu_2(W_c(Z^{l_a}))=l_ac.$ So,
$\mu_2(L_a(Z))=\mu_2(A)+\mu_2(W_c(Z^{l_a}))=\mu_2(A)+\mu_2(W_c(X^{l_a}))=\mu_2(L_a).$

Second, consider the case $a\in\left(a^{*},\mu_1(X) - u_1
\right]$. Observe that $M_{a^*,u_1} \subseteq L_{a^*+u_1} \subset
L_{a+u_1} = \left\{f(x)<l_{a+u_1}\right\} \cup
W_c\left(X^{l_{a+u_1}}\right)$, where $c = a + u_1 - \mu_1 \left(
\left\{ f(x) < l_{a+u_1} \right\} \right)$. In addition,
$\left\{f(x)<l_{a+u_1}\right\} \cup W_c\left(X^{l_{a+u_1}}\right)
\subseteq \left\{f(x) \le l_{a+u_1}\right\} $. Therefore,
\begin{eqnarray*}
&&\mu_1 \left( \left\{x \in Z: f(x) \le l_{a+u_1} \right\} \right)  \\
&=&  \mu_1 \left( \left\{ f(x) \le l_{a+u_1} \right\}  \cap Z \right)
= \mu_1 \left( \left\{ f(x) \le l_{a+u_1} \right\} \setminus
M_{a^*,u_1} \right) \\
&\ge& \mu_1 \left( \left\{f(x)<l_{a+u_1}\right\} \cup W_c \left(X^{l_{a+u_1}}\right) \setminus M_{a^*,u_1}  \right)
 =  a + u_1 - u_1 = a.
\end{eqnarray*}
Thus, (\ref{eq:l_a(Y)}) implies that $l_a(Z) \le l_{a+u_1}$. On
the other hand, note that $M_{a^*,u_1} \subseteq \left\{f(x) \le l_a(Z)\right\}$. Indeed, since $a>a^*$, we have $l_a(Z) \ge l_{a^*} (Z) = l_{a^*}$. Assume $l_{a^*} \le l_a(Z) < l_{a^*+u_1}$, then $\{ x \in Z: f(x) \le l_a (Z) \} = \{f(x)\le l_a(Z)\} \setminus M_{a^*,u_1} = L_{a^*}$, and $a = \mu_1 (\{ x \in Z: f(x) \le l_a (Z) \}) = \mu_1 (L_{a^*}) = a^*$, which is a contradiction. Therefore, $l_a(Z) \ge l_{a^*+u_1}$ and $M_{a^*,u_1} \subseteq \left\{f(x) \le l_{a^*+u_1}\right\} \subseteq \left\{f(x) \le l_a(Z)\right\}$. Thus,
$\{x\notin Z:\, f(x)\le l_a(Z)\}=\{x\in M_{{a^*},u_1}:\, f(x)\le
l_a(Z)\}=M_{{a^*},u_1}$ and
\[
\mu_1 \left( \left\{ f(x) \le l_a(Z) \right\} \right) = \mu_1
\left( \left\{x \in Z: f(x) \le l_a(Z) \right\} \right) +
\mu_1\left(M_{a^*,u_1}\right) \ge  a+u_1,
\]
where the last step follows from property (b) in Lemma \ref{lem:L_q}.
Formula (\ref{eq:lq}) implies that $l_a(Z) \ge l_{a+u_1}$.
Therefore, $l_a(Z) = l_{a+u_1}$.

Consider again the identity $L_{a+u_1}=\{f(x)<l_{a+u_1}\}\cup
W_c(X^{l_a+u_1}),$ where the sets in the union are disjoint and
$c=\left( a+u_1 \right) -\mu_1(\{f(x)<l_{a+u_1}\})$. Similarly, $L_a(Z)=\{x\in Z:\,f(x)<l_{a+u_1}\}\cup
W_b(Z^{l_a+u_1}),$ where $b=a - \mu_1(\{ x \in Z : f(x) < l_{a}(Z) \})$.
Since $l_a(Z) = l_{a+u_1}$ and $\{f(x)<l_{a+u_1}\}\supset M_{a^*,u_1}$, we have $b=a - \mu_1(\{ x \in Z : f(x) < l_{a+u_1} \}) = a - \mu_1(\{ f(x) < l_{a+u_1} \} \setminus M_{a^*,u_1}) = (a + u_1) - \mu_1(\{ f(x) < l_{a+u_1} \}) = c$.
Thus,
\begin{eqnarray*}
\mu_2(L_{a}(Z)) &=& \mu_2(\{x \in Z: f(x) < l_a(Z)\})  + \mu_2( W_b (Z^{l_a(Z)})) \\
&=& \mu_2(\{ x \in Z : f(x) < l_{a+u_1} \}) + l_{a+u_1} \mu_1( W_b (Z^{l_{a+u_1}})) \\
&=& \mu_2(\{ f(x) < l_{a+u_1} \}) - \mu_2( M_{a^*,u_1} )+ l_{a+u_1} \mu_1( W_c (X^{l_{a+u_1}}))\\
&=& \mu_2({L}_{a+u_1}) -\mu_2(M_{a^*,u_1} )
 =\mu_2({L}_{a+u_1} \setminus M_{a^*,u_1} ),
\end{eqnarray*}
where the second equality holds because $l_a(Z)=l_{a+u_1}$,
$f(x)=l_{a+u_1}$ for $x\in X^{l_{a+u_1}}$, and
$Z^{l_{a+u_1}}\subseteq X^{l_{a+u_1}}$ (in fact
$Z^{l_{a+u_1}}=X^{l_{a+u_1}}$, but we do not use this). The third
equality holds because of $\{x\in
Z:\,f(x)<l_{a+u_1}\}=\{f(x)<l_{a+u_1}\}\setminus M_{a^*,u_1}$,
$\{f(x)<l_{a+u_1}\}\supset M_{a^*,u_1}$, and $b=c.$
The fourth equality follows from $l_{a+u_1} \mu_1( W_c (X^{l_{a+u_1}})) = \mu_2 ( W_c (X^{l_{a+u_1}}))$.
\end{proof}

\begin{lemma} \label{lem:boundary}
Let Assumption \ref{ass:abscon} hold. For a given
$u=\left(u_1,u_2\right)\in R_\mu(X)$, consider
$a^*\in\left[0,\mu_1(X)-u_1\right]$ with $\mu\left(
{M}_{a^{*},u_{1}} \right) = u$. 
Let ${q}=\left(q_1,q_2\right)$ be on the
lower (upper) boundary of ${R}_\mu \left(Z\right)$. If $q_1\in
\left[0,a^*\right]$ ($q_1 \in \left[0, \mu_1(X)-u_1 -a^*
\right)$), then ${q}$ is on the lower (upper) boundary of
${R}_{{\mu}}\left({X}\right)$ and, if $q_1 \in \left(a^*,
\mu_1(X)-u_1 \right]$ ($q_1\in
\left[\mu_1(X)-u_1-a^*,\mu_1(X)-u_1\right]$), then
${r}={\mu}\left({X}\right)-{u}-{q}$ is on the upper (lower)
boundary of ${R}_{{\mu}}\left({X}\right)$.

\end{lemma}

\begin{proof}

When ${q}$ is on the lower boundary of ${R}_\mu \left(Z\right)$,
according to Lemma \ref{lem:onboundary}, $\mu_2(L_{q_1}(Z))=q_2$. If $q_{1}\in\left[0,a^{*}\right]$, then by Lemma \ref{lem:connection},
$\mu_2(L_{q_1})=\mu_2(L_{q_1}(Z))=q_2$, and Lemma
\ref{lem:onboundary} implies that ${q}$ is on the lower boundary of
${R}_{{\mu}}\left({X}\right)$.

If
$q_{1}\in\left(a^{*},{\mu}\left({X}\right)-u_{1}\right]$, then for
$r=\left(r_1,r_2\right)$
\begin{eqnarray*}
r_2 & = & \mu_2\left({X}\right)- u_2 - q_2
 = \mu_2\left({X}\right)-\mu_2\left({M}_{a^{*},u_{1}}\right)-\mu_2\left({L}_{q_{1}}(Z)\right)\\
 &=& \mu_2\left({X}\right)-\left(\mu_2\left({M}_{a^{*},u_{1}}\right)+\mu_2\left({L}_{q_{1}+u_{1}}\setminus{M}_{a^{*},u_{1}}\right)\right)
  =  \mu_2\left({X}\right)-\mu_2\left({L}_{q_{1}+u_{1}}\right)\\
 & = & \mu_2\left({X}\setminus{L}_{q_{1}+u_{1}}\right)=\mu_2\left({X}\setminus{L}_{\mu_{1}\left({X}\right)-r_{1}}\right),
\end{eqnarray*}
where the first and last equalities follow from the definition of
$r$, 
the second equality follows from Lemma \ref{lem:Z_star}, the third
equality follows from Lemma \ref{lem:connection}, and the fourth equality
follows from  $q_1>a^*$. According to Lemma \ref{lem:onboundary},
$r$ is on the upper boundary of ${R}_{{\mu}}\left({X}\right)$.

If ${q}$ is on the upper boundary of ${R}_\mu \left(Z\right)$, the, because of symmetry, $r = \mu(X) - u -q$ is on the lower boundary of ${R}_\mu \left(Z\right)$. If $q_1\in \left[\mu_1(X)-u_1-a^*,\mu_1(X)-u_1\right]$, then $\mu_1(X)-u_1-q_1 \in \left[0,a^*\right]$. From the first part of the proof, $r = \mu(X)-u-q$ is on the lower boundary of ${R}_{{\mu}}\left({X}\right)$. If $q_1\in \left[0, \mu_1(X)-u_1 -a^* \right)$, then $\mu_1(X)-u_1-q_1 \in \left(a^*, \mu_1(X)-u_1 \right]$. Again, from the first part of the proof, $\mu(X)-u-(\mu(X)-u-q_1) = q_1$ is on the upper boundary of ${R}_{{\mu}}\left({X}\right)$.
\end{proof}

\begin{lemma}\label{lem:weakthm}
Under Assumption \ref{ass:abscon}, for any vector ${p}\in{R}_{{\mu}}\left({X}\right)$, there exists a maximal set ${Z}^{*}\in\mathcal S_{{p}}\left({X}\right)$ and, in addition, ${R}^{{p}}_\mu\left({X}\right)=Q^{{p}}_{\mu}\left({X}\right).$
\end{lemma}

\begin{proof}
For $u=\mu(X)-p$, consider $a^*$ defined in Lemma \ref{lem:Z_star}. For $Z^*=X\setminus {M}_{a^{*},\mu_1(X)-p_{1}}$, the following three statements are true: (1) ${R}_{{\mu}}\left({Z}^{*}\right)\subseteq{R}^{{p}}_\mu\left({X}\right)$;
(2) ${R}^{{p}}_\mu\left({X}\right)\subseteq Q^{{p}}_{\mu}\left({X}\right)$;
(3) $Q^{{p}}_{\mu}\left({X}\right)\subseteq{R}_{{\mu}}\left({Z}^{*}\right)$.

For (1), $\mu\left(Z^{*}\right)=\mu\left({X}\setminus{M}_{a^{*},\mu_1(X)-p_{1}}\right)=\mu(X) - \mu\left({M}_{a^{*},\mu_1(X)-p_{1}}\right) = \mu(X) - \left(\mu(X) - p \right) = p$, where the second to the last equality follows from Lemma \ref{lem:Z_star}. Thus, $R_\mu \left(Z^*\right) = R_\mu ^p (X)$.

For (2), assume that there exists a vector ${q}\in{R}^{{p}}_\mu\left({X}\right)$
such that ${q}\notin Q^{{p}}_{\mu}\left({X}\right)$.
Then Definition \ref{def:sets} implies that either ${q}\notin{R}_{{\mu}}\left({X}\right)$
or ${q}\notin\left({R}_{{\mu}}\left({X}\right)-\left\{ {\mu}\left({X}\right)-{p}\right\} \right)$.
However $q\in R_\mu ^p (X) \subseteq R_\mu (X)$. Therefore, ${q}\notin{R}_{{\mu}}\left({X}\right)-\left\{ {\mu}\left({X}\right)-{p}\right\} $, which is equivalent to $p-q \notin \left\{\mu(X) \right\} - R_\mu(X) = R_\mu(X)$, where the equality follows from Lemma \ref{lem:censym}. Since ${R}^{{p}}_\mu(X) \subseteq R_\mu(X)$, we have $p-q \notin  R_\mu^p(X)$. By Lemma \ref{lem:censymsub}, $R_\mu^p(X)$ is centrally symmetric with the center $\frac{p}{2}$. Therefore $q\notin R_\mu^p(X)$. The above contradiction implies (2).

For (3), assume ${q}\in Q^{{p}}_{\mu}\left({X}\right)$,
but ${q}\notin{R}_{{\mu}}\left({Z}^{*}\right)$. By Lyapunov's theorem $R_\mu\left(Z^*\right)$ is a convex compactum. Let ${q}^{u}=\left(q_{1},q^u_{2}\right)$
and ${q}^{l}=\left(q_{1},q^l_{2}\right)$ be the intersection
points of the vertical line $\mu_{1}=q_{1}$ and the upper and lower boundaries
of ${R}_{{\mu}}\left({Z}^{*}\right)$ respectively. Then one
of the following must be true: $q_{2}>q^u_{2}$ or
$q_{2}<q^l_{2}$. Without loss of generosity, we consider
the former case. Since
$q_{u}$ is on the upper boundary of ${R}_{{\mu}}\left({Z}^{*}\right)$,
according to Lemma \ref{lem:boundary}, one of the following is true:
(a) ${q}^{u}$ is on the upper boundary of ${R}_{{\mu}}\left({X}\right)$
or (b) $r={p}-{q}^{u}$ is on the lower boundary
of ${R}_{{\mu}}\left({X}\right)$. For (a),
$q_{2}>q^u_{2}$ implies ${q}\notin{R}_{{\mu}}\left({X}\right)$.
Thus ${q}\notin Q^{{p}}_{\mu}\left({X}\right)$. This contradicts our assumption. For (b), we let $r'={p}-{q}$. Obviously, $r'_1=r_1$ and $r'_2<r_2$. This implies that $r'$ is below the lower boundary point $r$. Thus, $r' \notin{R}_{{\mu}}\left({X}\right)$ and $r' \notin Q^{{p}}_{\mu}\left({X}\right)$.
But according to Lemma \ref{lem:censymsub}, this means ${q}\notin Q^{{p}}_{\mu}\left({X}\right)$,
which contradicts to our assumption. Statement (1)-(3) imply the lemma.
\end{proof}

Let $D$ be a two-by-two invertible matrix with positive entries, and $A\subseteq \mathbb{R}^2$. We denote by $A D$ the set $ \left\{ p D: p\in A \right\}$. For a vector measure $\mu = \left(\mu_1,\mu_2\right)$, let $\nu=\mu D$ be the vector measure $\left(\nu_1,\nu_2\right) = \left(D_{11}\mu_1 + D_{21} \mu_2 , D_{12}\mu_1 + D_{22} \mu_2 \right)$. Then the measure $\nu_1$ and $\nu_2$ are equivalent.

\begin{lemma}\label{lem:mapping}
(a) $R_\mu(Y)D=R_\nu(Y)$ for all $Y \in \mathcal{F}$; (b) $R_\mu^p(X) D = R_\nu^{pD}(X)$ for all $p\in R_\mu(X)$; (c) $Q_\mu^p(X) D = Q_\nu^{pD}(X)$ for all $p\in R_\mu(X)$.
\end{lemma}

\begin{proof}
(a) For any point $q \in R_\nu(Y) $, there exists a set $Z\in \mathcal{F}_Y$ such that $\nu(Z)=q$. Since $\mu(Z)=q D^{-1}$ and $q D^{-1} \in R_\mu(Y)$, we have $q \in R_\mu(Y) D $. For any point $q \in R_\mu(Y) D $, we have $q D^{-1} \in R_\mu(Y)$. Thus there exists a set $Z\in \mathcal{F}_Y$ such that $\mu(Z)=q D^{-1}$, and $\nu(Z)=q$. Therefore, $\nu(Z) \in R_\nu(Y)$.

(b) For any point $q \in R_\nu^{p D}(X)$, there exist sets $Y\in\mathcal{F}$ and $Z\in\mathcal{F}_Y$ such that $\nu(Y) = p D$ and $\nu(Z) = q$. So $\mu(Y) = p$ and $\mu(Z) = q D^{-1}$. Thus, $q D^{-1} \in R_\mu^p(X)$ and therefore, $q \in R_\mu^p(X) D$. For any point $q \in R_\mu^p(X)D$, we have $q D^{-1} \in R_\mu^p(X)$. So there exist sets $Y\in\mathcal{F}$ and $Z\in\mathcal{F}_Y$ such that $\mu(Y) = p$ and $\mu(Z) = q D^{-1}$, and consequently $\nu(Y) = p D$ and $\nu(Z) = q$. Thus $q \in R_\mu^{p D}(X)$.

(c) According to Definition \ref{def:sets}, $Q_\mu^p(X) D = (R_\mu(X) D-\{\mu (X)D-{p}D\})\cap{R}_{{\mu}}(X)D = (R_\nu(X)-\{\nu (X)-{p}D\})\cap{R}_{{\nu}}(X) = Q_\nu^{pD}(X)$.
\end{proof}

\begin{proof}[Proof of Theorem \ref{thm1}]
According to Lemma \ref{lem:weakthm}, Theorem \ref{thm1} holds under Assumption \ref{ass:abscon} that states that $\mu_1$ and $\mu_2$ are equivalent. If $\mu_1$ and $\mu_2$ are not equivalent, consider $\nu=\mu D$. Since $\nu_1$ and $\nu_2$ are equivalent, $Q^p_\mu (X) =  Q^{pD}_{\nu}(X) D^{-1} = R_\nu^{pD}(X) D^{-1} = R_\mu^p(X)$,
where the first equality and the last equality is by Lemma \ref{lem:mapping}, and the second equality is due to Lemma \ref{lem:weakthm}.
Furthermore, according to Lemma \ref{lem:weakthm}, there exists a maximal set $Z^*$, such that $R_\nu\left(Z^*\right) = R_\nu^{pD}(X)$. Therefore, $R_\mu\left(Z^*\right) = R_\nu\left(Z^*\right) D^{-1} =  R^{pD}_{\nu}(X) D^{-1} = R^p_\mu (X)$.
\end{proof}

Now consider Theorem \ref{thm1.5}. For $A\subseteq \mathbb{R}^m$
and $b\in\mathbb{R}^m$, let $A+b = \left\{a+b:a\in A \right\}$ and
$A-b = A+(-b)$. Observe that $A-b=A-\{b\}$. Recall that $A\oplus B
= \bigcup_{b\in B} \left( A + b \right) $ is called the Minkowski
addition, and $A \ominus B = \bigcap_{b\in B} \left(A-b\right)$ is
called the Minkowski subtraction, where $A,B \subseteq
\mathbb{R}^m$.
%
\begin{lemma}\label{lem:seesaw}
Let $A_1,A_2,B_1,B_2 \subseteq \mathbb{R}^2$ be convex and compact sets such that $A_1 \oplus B_1 = A_2 \oplus B_2 $ and $B_1 \subseteq B_2$. Then $A_2 \subseteq A_1$.
\end{lemma}
\begin{proof}
According to \cite[Lemma 3.1.8]{Schneider:2008}, if $A,B\subseteq \mathbb{R}^2$ are convex and compact sets then $(A \oplus B) \ominus B = A$. Thus if $a\in A_2$, then $a \in \left( A_2 \oplus B_2 \right) \ominus B_2$, and consequently $a \in \left( A_1 \oplus B_1 \right) \ominus B_2$. So $a \in \left( A_1 \oplus B_1 \right)-b $, for any $b\in B_2$. Since $B_1 \subseteq B_2$, we have $a \in \left( A_1 \oplus B_1 \right)-b $, for any $b\in B_1$, and thus $a \in \left( A_1 \oplus B_1 \right) \ominus B_1 = A_1$.
\end{proof}

\begin{proof}[Proof of Theorem \ref{thm1.5}]
Now let $Z^*$ be the maximal set with the measure $p$, then $\mu\left(X\setminus Z^*\right) = \mu(X) - p $. Consider any set $M$, such that $\mu(M)=\mu(X)-p$. Obviously, $R_\mu(M) \oplus R_\mu(X \setminus M) = R_\mu\left(X \setminus Z^*\right) \oplus R_\mu\left( Z^* \right) = R_\mu(X)$. In addition, $R_\mu(X \setminus M) \subseteq R_\mu\left( Z^* \right)$ by definition. Thus according to Lemma \ref{lem:seesaw}, $R_\mu\left(X \setminus Z^*\right) \subseteq R_\mu(M)$.

Similarly, let $M^*=X \setminus Z^*$ be the minimal set with the measure $\mu(X)-p$, then $\mu\left(Z^*\right) = p $. Consider any set $Z$, such that $\mu(Z)=p$. Obviously, $R_\mu(Z) \oplus R_\mu{(X \setminus Z)} = R_\mu\left(Z^*\right) \oplus R_\mu\left( X \setminus  Z^* \right) = R_\mu(X)$. In addition, $R_\mu(X \setminus Z^*) = R_\mu(M^*) \subseteq R_\mu\left(X \setminus  Z \right)$ by definition. Thus according to Lemma \ref{lem:seesaw}, $R_\mu\left(Z\right) \subseteq R_\mu(Z^*)$.
\end{proof}

With Theorem \ref{thm1.5}, the existence of the minimal subset ${M}^{*}\in\mathcal S_\mu^{{q}}\left({X}\right)$ immediately follows from the existence of the maximal subset ${Z}^{*}\in\mathcal S_\mu^{\mu(X) - q}\left({X}\right)$. Furthermore, ${R}_\mu\left({M}^{*}\right)= \left( {R}_\mu\left({M}^{*}\right) \oplus {R}_\mu\left({Z}^{*}\right) \right)\ominus {R}_\mu\left({Z}^{*}\right) =  R_\mu(X) \ominus Q^{\mu(X) - {q}}_{\mu}\left({X}\right)$.

\begin{corollary}
For a two-dimensional finite atomless vector measure
${\mu}=\left(\mu_{1},\mu_{2}\right)$ and for a vector
${q}\in{R}_{{\mu}}\left({X}\right)$, there exists a minimal set
${M}^{*}\in\mathcal S_\mu^{{q}}\left({X}\right)$. In addition,
${R}_\mu\left({M}^{*}\right)=
R_\mu(X) \ominus Q^{\mu(X) - {q}}_{\mu}\left({X}\right)$.
\end{corollary}

\section{Counterexample for 3D measures}\label{s4}

In this section, we present an example of a measurable
space $(X,\mathcal{F})$ endowed with a three-dimensional atomless
finite measure $\nu=\left(\nu_{1},\nu_{2},\nu_{3}\right)$ and a vector
$p\in R_{\nu}(X)$ such that a maximal subset of $X$ with the measure $p$ does
not exist.
Theorem \ref{thm1.5} implies that the minimum set does not exist either in this example.

Recall that, with respect to a measure $\mu$, set $A$ and $B$ are
said to be equal up to null sets (denoted by $A\simeq B$) if $\mu\left(A\setminus B\right)=\mu\left(B\setminus A\right)=0$. Also recall that $X^l=\left\{f(x)=l\right\}$.

\begin{proposition}\label{lem:unique}
Let $\mu=\left(\mu_1,\mu_2\right)$ satisfy Assumption
\ref{ass:abscon} and let $Y \in \mathcal{F}$. If $\mu_{1}\left\{
X^{l_{\mu_{1}(Y)}}\right\} =0$ and
$\mu_{2}(Y)=\mu_{2}\left(L_{\mu_{1}(Y)}\right)$, then $Y\simeq
L_{\mu_{1}(Y)}$.
\end{proposition}
\begin{proof}
Assume that $Y\simeq L_{\mu_1(Y)}$ does not hold. We define three disjoint sets
${Z}_{1}={Y}\setminus{L}_{\mu_1(Y)}$,
${Z}_{2}={L}_{\mu_1(Y)}\setminus{Y}$, and $M={Y} \cap {L}_{\mu_1(Y)}$.
 Observe that $Y=Z_1\cup M$ and
 ${L}_{\mu_1(Y)}=Z_2\cup M$. These equalities and  $\mu_1(Y)=\mu_1(L_{\mu_1(Y)})$ imply $\mu_1(Z_1)=\mu_1(Z_2)$.
 Furthermore, $Z_1\subseteq \left\{f(x) \ge l_{\mu_1(Y)} \right\}$ and
 $Z_2\subseteq \left\{ f(x) <  l_{\mu_1(Y)} \right\}$, because according to (\ref{eq:L_a}), ${L}_{\mu_1(Y)}=
 \left\{ f\left(x\right)<l_{\mu_1(Y)}\right\}$ when $\mu_{1}\left\{ X^{l_{\mu_{1}(Y)}}\right\} =0$. Therefore,
\begin{eqnarray*}
\mu_{2}\left({Z}_{1}\right) &=& \int_{{Z}_{1}}f\left(x\right)\mu_{1}\left(dx\right)
 \ge l_{\mu_1(Y)}\int_{{Z}_{1}}\mu_{1}\left(dx\right) \\
 &=& l_{\mu_1(Y)}\int_{{Z}_{2}}\mu_{1}\left(dx\right)
 >  \int_{{Z}_{2}}f\left(x\right)\mu_{1}\left(dx\right)=\mu_{2}\left({Z}_{2}\right).
\end{eqnarray*}
So $\mu_{2}\left({Y}\right)= \mu_{2}\left({Z_1}\right) + \mu_{2}\left({M}\right) > \mu_{2} \left(Z_2\right)  + \mu_{2} (M)  = \mu_{2} \left({L}_{\mu_1(Y)}\right)$. This contradiction implies the proposition.
\end{proof}
\begin{example} \label{exa:counter}
Let $X=[0,1]$ and $\mathcal{F}$ is the Borel $\sigma$-field.
Consider the three-dimensional vector measure
${\nu}\left(dx\right)=\left(\nu_{1},\nu_{2},\nu_{3}\right)\left(dx\right)=\left(1,2x,\rho\left(x\right)\right)dx$,
where
\[
\rho \left(x\right)=\left\{ \begin{array}{ll}
4x, & \textrm{if } x\in\left[0,\frac{1}{2}\right);\\
4x-2, & \textrm{if } x\in\left[\frac{1}{2},1\right].\end{array}\right.
\]
%
%
Consider the points ${p}=\left(\frac{1}{2},\frac{1}{2},\frac{1}{2}\right)$,
${q^1}=\left(\frac{1}{4},\frac{1}{16},\frac{1}{8}\right)$, and ${q^2}=\left(\frac{1}{4},\frac{5}{32},\frac{1}{16}\right)$. It is easy to show that $q^1,q^2\in R_\nu^p(X)$. Indeed let $Z^1=\left[0,\frac{1}{4}\right) \cup\left[\frac{3}{4},1\right]$, $Z^2=\left[0,\frac{1}{8}\right)\cup\left[\frac{3}{8},\frac{5}{8}\right)\cup\left[\frac{7}{8},1\right]$, $W^1=\left[0,\frac{1}{4}\right)\subseteq Z^1$, and $W^2=\left[0,\frac{1}{8}\right)\cup \left[\frac{1}{2},\frac{5}{8}\right)\subseteq Z^2$, and we have $\nu\left(Z^1\right)=\nu\left(Z^2\right)={p}$, $\nu\left(W^1\right)=q^1$, and $\nu\left(W^2\right)=q^2$. Since $Z^1$ and $Z^2$ are not equal up to a null set, Proposition \ref{pro:Z} implies that there doesn't exist a set $Z$ such that $\nu(Z)=p$ and $q^1,q^2\in R_\mu(Z)$.
\end{example}
\begin{proposition}\label{pro:Z}
Consider the sets $X$, $Z^1$, $Z^2$, the measure $\nu$ and vectors $p$, $q^1$, $q^2$ from Example \ref{exa:counter}. Let $Z \in \mathcal{S}_\nu^p(X)$. For each $i=1,2$, if $q^i \in R_\nu(Z)$, then $Z \simeq Z^i$.
\end{proposition}
\begin{proof}
Let $i =1$. Since $q^1 \in R_\nu(Z)$, there exists a set ${W^1}
\in \mathcal{F}_Z$ such that $\nu({W^1}) = q^1$. Define
two-dimensional vector measure
$\mu=\left(\mu_1,\mu_2\right)=\left(\nu_1,\nu_2\right)$. Then
$\mu({W^1}) = \left(\frac{1}{4},\frac{1}{16}\right)$. Observe
that, according to (\ref{eq:lq}) and (\ref{eq:L_a}),
$l_{\mu_1({W^1})}=l_{\frac{1}{4}}=\frac{1}{2}$ and
$L_{\mu_1({W^1})}=L_{\frac{1}{4}}=\left[0,\frac{1}{4}\right)$. In
addition, $\mu_1\left(X^{l_{\mu_1({W^1})}}\right)=0$ and
$\mu_{2}({W^1})=\frac{1}{16}=\mu_{2}\left(L_{\mu_{1}({W^1})}\right)$.
Therefore, according to Proposition~\ref{lem:unique}, ${W^1}
\simeq L_{\mu_{1}({W^1})}=\left[0,\frac{1}{4}\right)$. On the
other hand, let $Y={W^1}\cup (X\setminus Z)$. Since ${W^1}
\subseteq Z$,
$\nu(Y)=\nu({W^1})+(\nu(X)-\nu(Z))=q^1+(\nu(X)-p)=\left(\frac{3}{4},\frac{9}{16},\frac{5}{8}\right)$,
and thus, $\mu(Y)=\left(\frac{3}{4},\frac{9}{16}\right)$. Observe
that, according to (\ref{eq:lq}) and (\ref{eq:L_a}),
$l_{\mu_1(Y)}=l_{\frac{3}{4}}=\frac{3}{2}$ and
$L_{\mu_1(Y)}=L_{\frac{3}{4}}=\left[0,\frac{3}{4}\right)$. In
addition, $\mu_1\left(X^{l_{\mu_1(Y)}}\right)=0$ and
$\mu_{2}(Y)=\frac{9}{16}=\mu_{2}\left(L_{\mu_{1}(Y)}\right)$.
Therefore, according to Proposition \ref{lem:unique}, $Y \simeq
L_{\mu_{1}(Y)}=L_{\frac{3}{4}}$.  Above observations imply that $Z
= {W^1} \cup (X\setminus Y) \simeq L_{\frac{1}{4}} \cup
\left(X\setminus L_{\frac{3}{4}}\right) =
\left[0,\frac{1}{4}\right) \cup \left(\left[0,1\right]\setminus
\left[0,\frac{3}{4}\right)\right) =  \left[0,\frac{1}{4}\right)
\cup \left[\frac{3}{4},1\right]=Z^1$.

Let $i =2$. Since $q^2 \in R_\nu(Z)$, there exists a set ${W^2}
\in \mathcal{F}_Z$ such that $\nu({W^2}) = q^2$. Define
two-dimensional vector measure
$\mu=\left(\mu_1,\mu_2\right)=\left(\nu_1,\nu_3\right)$. Then
$\mu({W^2}) = \left(\frac{1}{4},\frac{1}{16}\right)$. Observe
that, according to (\ref{eq:lq}) and (\ref{eq:L_a}),
$l_{\mu_1({W^2})}=l_{\frac{1}{4}}=\frac{1}{2}$ and
$L_{\mu_1({W^2})}=L_{\frac{1}{4}}=\left[0,\frac{1}{8}\right) \cup
\left[\frac{1}{2},\frac{5}{8}\right)$. In addition,
$\mu_1\left(X^{l_{\mu_1({W^2})}}\right)=0$ and
$\mu_{2}({W^2})=\frac{1}{16}=\mu_{2}\left(L_{\mu_{1}({W^2})}\right)$.
Therefore, according to Proposition~\ref{lem:unique}, ${W^2}
\simeq L_{\mu_{1}({W^2})}=\left[0,\frac{1}{8}\right) \cup
\left[\frac{1}{2},\frac{5}{8}\right)$. On the other hand, let
$Y={W^2} \cup (X\setminus Z)$. Since ${W^2} \subseteq Z$,
$\nu(Y)=\nu({W^2})+(\nu(X)-\nu(Z))=q^2+(\nu(X)-p)=\left(\frac{3}{4},\frac{21}{32},\frac{9}{16}\right)$,
and thus, $\mu(Y)=\left(\frac{3}{4},\frac{9}{16}\right)$. Observe
that, according to (\ref{eq:lq}) and (\ref{eq:L_a}),
$l_{\mu_1(Y)}=l_{\frac{3}{4}}=\frac{3}{2}$ and
$L_{\mu_1(Y)}=L_{\frac{3}{4}}=\left[0,\frac{3}{8}\right) \cup
\left[\frac{1}{2},\frac{7}{8}\right)$. In addition,
$\mu_1\left(X^{l_{\mu_1(Y)}}\right)=0$ and
$\mu_{2}(Y)=\frac{9}{16}=\mu_{2}\left(L_{\mu_{1}(Y)}\right)$.
Therefore, according to Proposition \ref{lem:unique}, $Y \simeq
L_{\mu_{1}(Y)}=L_{\frac{3}{4}}$.  Above observations imply that $Z
= {W^2} \cup (X\setminus Y) \simeq L_{\frac{1}{4}} \cup
\left(X\setminus L_{\frac{3}{4}}\right)
=\left(\left[0,\frac{1}{8}\right) \cup
\left[\frac{1}{2},\frac{5}{8}\right)\right) \cup
\left(\left[0,1\right]\setminus \left(\left[0,\frac{3}{8}\right)
\cup \left[\frac{1}{2},\frac{7}{8}\right)\right)\right) =
\left[0,\frac{1}{8}\right)\cup\left[\frac{3}{8},\frac{5}{8}\right)\cup\left[\frac{7}{8},1\right]=Z^2$.
\end{proof}
\section{Geometric construction of maximal ranges}\label{s5}

In \cite{Liapounoff:1940}, Lyapunov commented that a subset of the
two-dimensional Euclidean space $\mathbb{R}^2$ is the range of
some two-dimensional finite atomless vector measure on some
measurable space if and only if it satisfies the following
conditions: (1) it is convex; (2) it is closed; (3) it is
centrally symmetric; (4) it contains the origin. Since the
geometrically constructed set $Q_\mu^p(X)$ satisfies the
conditions (1)-(4), it must be the range of some two-dimensional
finite atomless vector measure on some measurable space.
Theorem~\ref{thm1} immediately tells us that it is the range of
the vector measure $\mu$ on the measurable space
$\left(Z^*,\mathcal{F}_{Z^*}\right)$. The second equality in
Theorem~\ref{thm1} allows us to construct geometrically the set
$R_\mu^p(x)$ by shifting the set $R_\mu(X)$ by
$\left(p-\mu(X)\right)$ and intersecting the shifted set with
$R_\mu(X)$.

We consider three examples with the same set $X=[0,1]$, but with
different probability vector measures. Let $p=(0.7,0.8)$ in all
these examples.
%
\begin{figure}[ht]
\begin{center}
\includegraphics[scale=0.55]{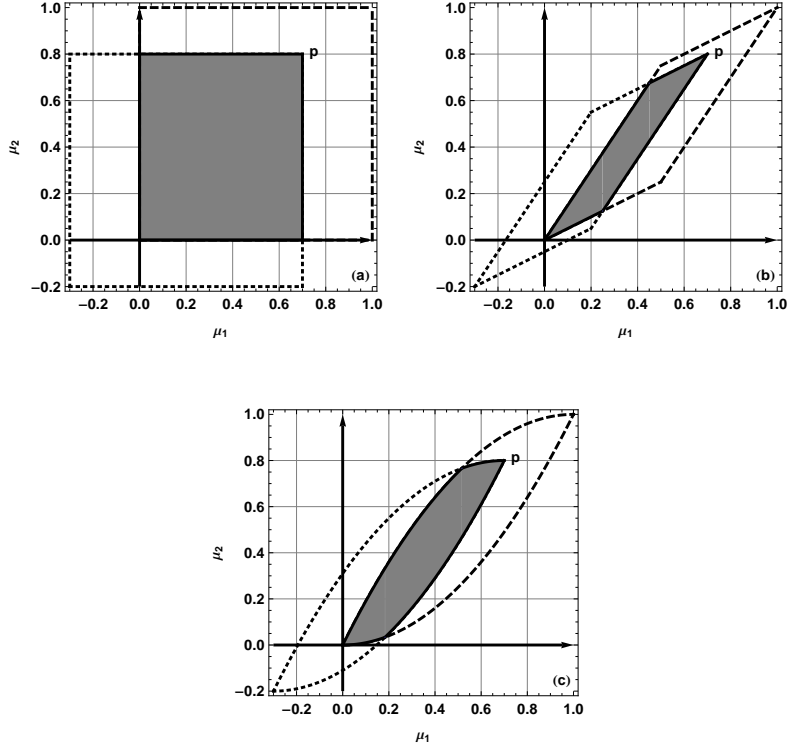}\label{fig:FigTot}
\caption{Plots (a)-(c) present the maximal subsets for the vector measures described in Examples \ref{Exa:Singular}-\ref{Exa:Linear}, respectively, with $p=(0.7,0.8)$. The area enclosed by the dashed lines is the range $R_\mu(X)$.
The area enclosed by the dotted lines are obtained by parallelly shifting
the dashed area by $\left(-0.3,-0.2\right)$. The shaded areas
are intersection of the above two areas and represents the identical
sets ${R}_\mu\left({Z}^{*}\right)$, ${R}^{{p}}_\mu\left({X}\right)$
and $Q^{{p}}_{\mu}\left({X}\right)$.}
\end{center}
\end{figure}
\begin{example}\label{Exa:Singular}
Let $\mu_{1}$ and $\mu_{2}$ be singular. Then the range $R_\mu(X)$ is the unit square enclosed by the dashed lines in Fig.~1(a). The shaded area denotes the identical sets ${R}_\mu\left({Z}^{*}\right)$,
${R}^{{p}}_\mu\left({X}\right)$ and $Q^{{p}}_{\mu}\left({X}\right)$
with ${p}=\left(0.7,0.8\right)$.
\end{example}

\begin{example}\label{Exa:Constant}
Consider the vector measure ${\mu}\left(dx\right)=\left(\mu_{1},\mu_{2}\right)\left(dx\right)=\left({1,f\left(x\right)}\right)dx$, where
\[
f\left(x\right)=\begin{cases}
\frac{1}{2}, & \textrm{if } x\in\left[0,\frac{1}{2}\right);\\
\frac{3}{2}, & \textrm{if } x\in\left[\frac{1}{2},1\right].\end{cases}
\]
Then the range of $R_\mu(X)$ is the area enclosed by
the dashed lines in Fig.~1(b). The shaded area denotes the three
identical sets ${R}_\mu\left({Z}^{*}\right)$, ${R}^{{p}}_\mu\left({X}\right)$
and $Q^{{p}}_{\mu}\left({X}\right)$ with ${p}=\left(0.7,0.8\right)$.
\end{example}

\begin{example}\label{Exa:Linear}
Let ${\mu}\left(dx\right)=\left(\mu_{1},\mu_{2}\right)\left(dx\right)=\left({1,2x}\right)dx$.
Then the range $R_{\mu}(X)$ is the area enclosed by
the dashed lines in Fig.~1(c). The shaded area denotes the three
identical sets ${R}_\mu\left({Z}^{*}\right)$, ${R}^{{p}}_\mu\left({X}\right)$,
and $Q^{{p}}_{\mu}\left({X}\right)$ when ${p}=\left(0.7,0.8\right)$.
%
%
%
%
%
%
%
\end{example}

\section{Proof of Theorem~\ref{thm2}}\label{s6}

 For any $B \subseteq A$, denote
$p(B)=\sum_{a \in B}p^a$, where either  $A=\{1,2,\ldots\}$ or $A=\{1,\ldots,n\}$ for some $n=1,2,\ldots\ .$  
\begin{lemma} \label{lem:PartialSum}
Let $\mu=\left(\mu_1,\mu_2\right)$ be a two-dimensional finite
atomless measure. If $p(B) \in R_\mu(X)$ for all $B \subset A$ and
$\sum_{a \in A} {p^a} = \mu (X)$, then there exists a partition
$\{Z^a\in {\mathcal{F}}:a\in A\}$ of $X$, such that $p^a=\mu
(Z^a)$ for each $a\in{A}$.
\end{lemma}
\begin{proof}
Consider $p=\mu(X)-p^1$. According to Theorem \ref{thm1}, there
exists a maximal subset $Z^*\in\mathcal{S}_\mu^{p}(X)$ and
$R_\mu\left(Z^*\right) = Q_\mu^{p}(X)$.  Let $Z^1 = X\setminus
Z^*$, $X^1=Z^*$, and $A^1 = A \setminus \{1\}$. Note that
$p^1=\mu(Z^1)$ and $p(B) \in R_\mu\left(X^1\right)$ for all $B
\subseteq A^1$. Indeed, ${p}(B) + {p^1}=p(B\cup \{1\}) \in R_\mu
(X)$.  Thus, $p(B) \in R_\mu (X) - \{(\mu(X) - p)\}$, and in
addition ${p}(B) \in R_\mu (X)$. Therefore, $p(B) \in Q_\mu^{p}(X)
= R_\mu\left(X^1\right)$.

Now for $p^2 \in \left\{p^a:a \in A^1\right\}$  there exists a
maximal set $Z^*\in\mathcal{S}_\mu^{p}(X^1)$, where
$p=\mu(X^1)-p^2$. Let $Z^2 = X^1 \setminus Z^*$, $X^2=Z^*$, and
$A^2 = A^1 \setminus \{2\}$, then $p^2=\mu(Z^2)$ and $p(B) \in
R_\mu\left(X^2\right)$ for all $B \subseteq A^2$. The repetition
of this procedure generates the desired partition
$\left\{Z^a\in\mathcal{F}: a \in A\right\}$.
\end{proof}

\begin{proof}[Proof of Theorem 2.5]
The necessity is obvious. For the sufficiency, in view of Lemma
\ref{lem:PartialSum}, it is sufficient to prove that condition
(ii) implies $p (B) \in R_\mu (X)$ for all $B\subseteq A$.  If $B$
is finite, condition (ii) implies $p(B) \in R_\mu (X)$. If $B$ is
infinite, let $B = \left\{a^1,a^2,\dots \right\}$ and $B_n =
\left\{a^1,a^2,\dots a^n\right\}$, $n=1,2,\ldots\ .$ Then $p(B) =
\lim_{n \rightarrow \infty} { p(B_n)}$ and $p(B_n) \in R_\mu (X)$
for $n=1,2,\dots$, according to condition (ii). Since $R_\mu (X)$
is closed, $p(B) \in R_\mu (X)$.
\end{proof}
%
%

Finally we show that, when $m=2$, the Dvoretzky-Wald-Wolfowitz
purification theorem for a countable image set $A$
\cite{Edwards:1987,Khan:2009}  is a particular case of
Theorem~\ref{thm2}. Let $p^a=\int_{{X}}\pi\left(a|x\right)\mu
\left(dx\right)$, $a\in A$. If these vectors $p^a$ satisfy
conditions (i) and (ii) of Theorem~\ref{thm2}, then
Theorem~\ref{thm2} implies that transition probability can be
purified in the case of countable  $A$ and $m=2$. Indeed, for (i),
obviously $\sum_{a \in A} p_a = \mu(X)$. For (ii),  if $B
\subseteq A$ then
\[
\sum_{a \in B} {p^a}={\sum_{a \in B} {\int_{{X}}\pi\left(a|x\right)\mu \left(dx\right)}}
=\int_{{X}}\pi\left(B|x\right)  \mu \left(dx\right)\in R_\mu(X),
\]
where the inclusion follows from a version of Lyapunov's theorem
\cite[p. 218]{Barra:1981}.

\bibliographystyle{amsplain}

\end{document}